\documentclass[10pt]{article}
\usepackage{float}
\usepackage{amssymb}
\usepackage{amsmath}
\usepackage{graphicx}
\usepackage{color}
\usepackage{ifpdf}
\DeclareMathOperator{\sgn}{sgn}
\newtheorem{theorem}{\scshape \mdseries Theorem}[section]
\newtheorem{corollary}[theorem]{\scshape \mdseries Corollary}
\newtheorem{definition}[theorem]{\scshape \mdseries Definition}
\newtheorem{lemma}[theorem]{\scshape \mdseries Lemma}

\newenvironment{proof}[1][Proof]{\noindent \textbf{#1.} }{\hfill  \rule{0.5em}{0.5em}}
\def \v{\mathbf{v}}
\def \u{\mathbf{u}}
\def \x{\mathbf{x}}

\def \i{\mathbf{i}}
\def \A{\mathcal{A}}
\def \L{\mathcal{L}}
\def \Q{\mathcal{Q}}
\def \T{\mathcal{T}}

\def \lamin{\lambda_{\min}}
\def \lamax{\lambda_{\max}}
\def \la{\lambda}
\def \Go{G^{\large \mathtt  o}}
\begin{document}

\title{\sf The $H$-spectrum of a generalized power hypergraph\thanks{Supported by the National Natural Science Foundation of China (11371028),
Scientific Research Fund for Fostering Distinguished Young Scholars of Anhui University(KJJQ1001),
Academic Innovation Team of Anhui University Project (KJTD001B).}}
\author{Murad-ul-Islam Khan, Yi-Zheng Fan\thanks{Corresponding author. E-mail addresses: fanyz@ahu.edu.cn(Y.-Z. Fan),
muradulislam@foxmail.com (M. Khan)} \\
{\small \it School of Mathematical Sciences, Anhui University, Hefei 230601, P. R. China} \\
}
\date{}
\maketitle

\textbf{Abstract}:
The generalized power of  a simple graph $G$, denoted by $G^{k,s}$, is obtained from $G$ by blowing up each vertex into an $s$-set and each edge
into a $k$-set, where $1 \le s \le \frac{k}{2}$. When $s < \frac{k}{2}$,   $G^{k,s}$ is always odd-bipartite.
It is known that $G^{k,{k \over 2}}$ is non-odd-bipartite if and only if $G$ is non-bipartite, and $G^{k,{k \over 2}}$ has the same adjacency (respectively, signless Laplacian) spectral radius as $G$.
In this paper, we prove that,  regardless of multiplicities,  the $H$-spectrum of $\A(G^{k,\frac{k}{2}})$ (respectively, $\Q(G^{k,\frac{k}{2}})$)
 consists of all eigenvalues of the adjacency matrices (respectively, the signless Laplacian matrices) of the connected induced subgraphs (respectively, modified induced subgraphs) of $G$.
 As a corollary, $G^{k,{k \over 2}}$ has the same least adjacency (respectively, least signless Laplacian) $H$-eigenvalue as $G$.
We also discuss the limit points of the least adjacency $H$-eigenvalues of hypergraphs, and construct a sequence of non-odd-bipartite hypergraphs whose
least adjacency $H$-eigenvalues converge to $-\sqrt{2+\sqrt{5}}$.

\textbf{Keywords}: Hypergraph; adjacency tensor; signless Laplacian tensor; least eigenvalue; limit point

\section{Introduction}
A {\it hypergraph} $G=(V(G),E(G))$ consists of  a set of vertices, say $V(G)=\{v_1,v_2,\ldots,v_n\}$, and a set of edges, say $E(G)=\{e_{1},e_2,\ldots,e_{m}\}$, where $e_{j}\subseteq V(G)$.
If $|e_{j}|=k$ for each $j=1,2,\ldots,m$, then $G$ is called a {\it $k$-uniform} hypergraph.
In particular, the $2$-uniform hypergraphs are exactly the classical simple graphs.
For a $k$-uniform hypergraph $G$, if we add to $G$ some edges with cardinality less than $k$, the resulting hypergraph denoted by $\Go$ is one with loops;
and those added edges are called the {\it loops} of $\Go$.
The {\it degree} $d_v(G)$ of a vertex $v \in V(G)$ is defined as $d_v(G)=|\{e_{j}:v\in e_{j}\in E(G)\}|$.
Sometimes we simply write $d_v(G)$ as $d_v$ if there exists no confusion.
So, a loop contributes $1$ to the degree of the vertex it is attached to.
 A {\it walk} $W$ of length $l$ in $G$ is a sequence of alternate vertices and edges: $v_{0},e_{1},v_{1},e_{2},\ldots,e_{l},v_{l}$,
    where $\{v_{i},v_{i+1}\}\subseteq e_{i}$ for $i=0,1,\ldots,l-1$.
The hypergraph $G$  is {\it connected} if every two vertices of $G$ are connected by a walk.

In recent years spectral hypergraph theory  has emerged as an important field in algebraic graph theory.
Let $G$ be a $k$-uniform hypergraph.
The {\it adjacency tensor} $\mathcal{A}=\mathcal{A}(G)=(a_{i_{1}i_{2}\ldots i_{k}})$ of $G$ is a $k$th order $n$-dimensional symmetric tensor,
  where $a_{i_{1}i_{2}\ldots i_{k}}=\frac{1}{(k-1)!}$ if $\{v_{i_{1}},v_{i_{2}},\ldots,v_{i_{k}}\} \in E(G)$, and $a_{i_{1}i_{2}\ldots i_{k}}=0$ otherwise.
   Let $\mathcal{D}=\mathcal{D}(G)$ be a $k$th order $n$-dimensional diagonal tensor,
    where $d_{i\ldots i}=d_{v_i}(G)$ for all $i \in [n]:=\{1,2,\ldots,n\}$.
Then $\L=\L(G)=\mathcal{D}(G)-\A(G)$ is the {\it Laplacian tensor} of the hypergraph $G$,
and $\Q=\Q(G)=\mathcal{D}(G)+\A(G)$ is the {\it signless Laplacian tensor} of $G$.
If $k=2$, the above tensors are the classical matrices of simple graphs.
The spectral radius (or the least $H$-eigenvalue) of the adjacency, Laplacian and signless Laplacian tensor of $G$ are denoted respectively by
$\rho^\A(G),\rho^\L(G),\rho^\Q(G)$ (or  respectively by $\lamin^\A(G),  \lamin^\L(G), \lamin^\Q(G)$).

For a hypergraph $\Go$ with loops, the adjacency tensor of $\Go$ is defined as the same as that of $G$, i.e. $\A(\Go)=\A(G)$.
The Laplacian tensor and the signless Laplacian tensor are defined by $\L(\Go)=\mathcal{D}(\Go)-\A(G)$ and $\Q(\Go)=\mathcal{D}(\Go)+\A(G)$, respectively.
So, even if $\Go$ is not uniform, the adjacency, Laplacian and signless Laplacian tensor of $\Go$ are all $k$th order $n$-dimensional tensors.

The spectral radius (or the largest $H$-eigenvalue) of the adjacency or signless Laplacian tensor of a uniform hypergraph has enjoyed a lot of research exposure;
see \cite{CD,HQX,HQS,LM,PZ,Qi,SSW}.
However, the least $H$-eigenvalue received little attention.
Nikiforov \cite{niki} gave a lower bound of $\lamin^\A(G)$ for an even uniform hypergraph $G$ in terms of order and size.
In fact he reduced the problem to discussing an odd-bipartite hypergraph; see \cite[Theorem 8.1]{niki}.
Here an  even uniform hypergraph $G$ is called {\it odd-bipartite} if $V(G)$ has a bipartition $V(G)=V_{1}\cup V_{2} $
   such that each edge has an odd number of vertices in both $V_{1}$ and $V_{2}$.
Shao et al. \cite{SSW} proved that the adjacency $H$-spectrum (or the adjacency spectrum) of $G$ is symmetric with respect to the origin if and only if $k$ is even and $G$ is odd-bipartite.
So, if $G$ is an odd-bipartite even uniform hypergraph, then $\lamin^\A(G)=-\rho^\A(G)$.

Qi \cite{Qi} showed that $\rho^\L(G) \le \rho^\Q(G)$, and posed a question of identifying the conditions under which the equality holds.
Hu et al. \cite{HQX} proved that if $G$ is connected, then the equality holds if and only if $k$ is even and $G$ is odd-bipartite.
Shao et al. \cite{SSW} proved a stronger result that the Laplacian $H$-spectrum (respectively, Laplacian spectrum) and the signless Laplacian $H$-spectrum (respectively, signless Laplacian spectrum) of a connected $k$-uniform hypergraph $G$ are equal if and only if $k$ is even and $G$ is odd-bipartite.
So, for an even $k$ and a connected $k$-uniform hypergraph $G$, if $G$ is odd-bipartite, then $\lamin^\Q(G)=\lamin^\L(G)=0$.

So, if we discuss the least $H$-eigenvalue of the adjacency or signless Laplacian tensor of a connected even uniform hypergraph,
it suffices to consider non-odd-bipartite hypergraphs.
Up to now, most known examples of hypergraphs are odd-bipartite.
Hu, Qi, Shao \cite{HQS} introduced the {\it cored hypergraphs} and the {\it power hypergraphs},
   where the cored hypergraph is one such that each edge contains at least one vertex of degree $1$,
   and the $k$th power of a simple graph $G$, denoted by $G^k$, is obtained by replacing each edge (a $2$-set) with a $k$-set by adding $k-2$ new vertices.
These two kinds of hypergraphs are both odd-bipartite.
Peng \cite{P} introduced $s$-paths and $s$-cycles, which are both $k$-uniform hypergraphs.
An $s$-path is always odd-bipartite \cite{KF}. But this does not hold for $s$-cycles.
When $1\leq s<\frac{k}{2}$, an $s$-cycle is odd-bipartite; and when $s=k/2$, it is odd-bipartite if and only if it has an even length.

We \cite{KF} introduced a generalized power hypergraph $G^{k,s}$ from a simple graph $G$, where $1 \le s \le k/2$.
If $s<k/2$, then $G^{k,s}$ is odd-bipartite; and $G^{k,k/2}$ ($k$ being even) is non-odd-bipartite if and only if $G$ is non-bipartite \cite{KF}.
So we can construct non-odd-bipartite hypergraphs from non-bipartite simple graphs.
In the paper \cite{KF}, we proved that $\rho^\A(G)=\rho^\A(G^{k,{\frac{k}{2}}})$ and $\rho^\Q(G)=\rho^\Q(G^{k,{\frac{k}{2}}})$.
In Section 3 we give a relationship between the spectrum of $\A(G)$ (respectively, $\Q(G)$) and the $H$-spectrum of $\A(G^{k,{\frac{k}{2}}})$ (respectively, $\Q(G^{k,{\frac{k}{2}}})$).
That is, regardless of multiplicities,  the $H$-spectrum of $\A(G^{k,\frac{k}{2}})$ (respectively, $\Q(G^{k,\frac{k}{2}})$)
 consists of all eigenvalues of the adjacency matrices (respectively, the signless Laplacian matrices) of the connected induced subgraphs (respectively, modified induced subgraphs) of $G$; see Theorem \ref{main2}.
 As a corollary, we show that $\lamin^\A(G)=\lamin^\A(G^{k,{\frac{k}{2}}})$ and $\lamin^\Q(G)=\lamin^\Q(G^{k,{\frac{k}{2}}})$.
We also discuss the limit points of the least adjacency $H$-eigenvalues of hypergraphs, and construct a sequence of non-odd-bipartite hypergraphs whose
least adjacency $H$-eigenvalues converge to $-\sqrt{2+\sqrt{5}}$.

\section{Preliminaries}
For integers $k\geq 1$ and $n\geq 2$,
  a real {\it tensor} (also called {\it hypermatrix}) $\mathcal{T}=(t_{i_{1}\ldots i_{k}})$ of order $k$ and dimension $n$ refers to a
  multidimensional array with entries $t_{i_{1}\ldots i_{k}}$ such that $t_{i_{1}\ldots i_{k}}\in \mathbb{R}$ for all $i_{j}\in [n]$ and $j\in [k]$.
 The tensor $\mathcal{T}$ is called \textit{symmetric} if its entries are invariant under any permutation of their indices.
 A {\it subtensor} of $\mathcal{T}$ is a multidimensional array with entries $t_{i_{1}\ldots i_{k}}$ such that $i_j \in S_j \subseteq [n]$ for some $S_j$'s and $j \in [k]$,
denoted by $\mathcal{T}[S_1|S_2|\cdots|S_k]$.
If $S_1=S_2=\cdots=S_k=:S$, then we simply write $\mathcal{T}[S_1|S_2|\cdots|S_k]$ as $\mathcal{T}[S]$, which is called the {\it principal subtensor} of $\mathcal{T}$.
If $k=2$, then $\mathcal{T}[S]$ is exactly the principal submatrix of $\mathcal{T}$;
and if $k=1$, then $\mathcal{T}[S]$ is the subvector of $\mathcal{T}$.

 Given a vector $x\in \mathbb{R}^{n}$, $\mathcal{T}x^{k}$ is a real number, and $\mathcal{T}x^{k-1}$ is an $n$-dimensional vector, which are defined as follows:
   $$\mathcal{T}x^{k}=\sum_{i_1,i_{2},\ldots,i_{k}\in [n]}t_{i_1i_{2}\ldots i_{k}}x_{i_1}x_{i_{2}}\cdots x_{i_k},~
   (\mathcal{T}x^{k-1})_i=\sum_{i_{2},\ldots,i_{k}\in [n]}t_{ii_{2}\ldots i_{k}}x_{i_{2}}\cdots x_{i_k} \mbox{~for~} i \in [n].$$
 Let $\mathcal{I}$ be the {\it identity tensor} of order $k$ and dimension $n$, that is, $i_{i_{1}i_2 \ldots i_{k}}=1$ if and only if
   $i_{1}=i_2=\cdots=i_{k} \in [n]$ and $i_{i_{1}i_2 \ldots i_{k}}=0$ otherwise.

\begin{definition}{\em \cite{Qi2}} Let $\mathcal{T}$ be a $k$th order n-dimensional real tensor.
For some $\lambda \in \mathbb{C}$, if the polynomial system $(\lambda \mathcal{I}-\mathcal{T})x^{k-1}=0$, or equivalently $\mathcal{T}x^{k-1}=\lambda x^{[k-1]}$, has a solution $x\in \mathbb{C}^{n}\backslash \{0\}$,
then $\lambda $ is called an eigenvalue of $\mathcal{T}$ and $x$ is an eigenvector of $\mathcal{T}$ associated with $\lambda$,
where $x^{[k-1]}:=(x_1^{k-1}, x_2^{k-1},\ldots,x_n^{k-1}) \in \mathbb{C}^n$.
\end{definition}

If $x$ is a real eigenvector of $\mathcal{T}$, surely the corresponding eigenvalue $\lambda$ is real.
In this case, $x$ is called an {\it $H$-eigenvector} and $\lambda$ is called an {\it $H$-eigenvalue}.
The {\it $H$-spectrum} of $\mathcal{T}$ is the multi-set of $H$-eigenvalues of $\mathcal{T}$ (counting multiplicities).
Furthermore, if $x\in \mathbb{R}_{+}^{n}$ (the set of nonnegative vectors of dimension $n$), then $\lambda $ is called an {\it $H^{+}$-eigenvalue} of $\mathcal{T}$;
if $x\in \mathbb{R}_{++}^{n}$ (the set of positive vectors of dimension $n$), then $\lambda$ is said to be an {\it $H^{++}$-eigenvalue} of $\mathcal{T}$.
The smallest $H$-eigenvalue and the largest $H$-eigenvalue of $\T$ are denoted by $\lamin(\T)$ and $\lamax(\T)$, respectively.
If an eigenvector $x$ of $\mathcal{T}$ cannot be scaled to be real, then it is called an {\it $N$-eigenvector}.
The {\it spectral radius of $\T$} is defined as
$$\rho(\T)=\max\{|\lambda|: \lambda \mbox{ is an eigenvalue of } \T \}.$$
Surely, by Theorem \ref{PF} presented later on, if $\T$ is nonnegative, then $\rho(\T)=\lamax(\T)$.

Chang et al. \cite{CPZ} introduced the irreducibility of tensors. A tensor $\T=(t_{i_{1}...i_{k}})$ of order $k$ and dimension $n$ is called {\it reducible} if there exists a nonempty proper subset $I \subsetneq [n]$ such that
$t_{i_1i_2\ldots i_k}=0$ for any $i_1 \in I$ and any $i_2,\ldots,i_k \notin I$.
If $\T$ is not reducible, then it is called {\it irreducible}.
Friedland et al. \cite{FGH} proposed a weak version of irreducibility.
The graph associated with $\T$, denoted by $G(\T)$, is the directed graph with vertices $1, \ldots, n$ and an edge from $i$ to $j$
 if and only if $t_{ii_2\ldots i_k}>0$ for some $i_l = j$, $l \in \{ 2, 3, \ldots, k\}$.
The tensor $\T$ is called {\it weakly irreducible} if $G(\T)$ is strongly connected.
Surely, an irreducible tensor is always weakly irreducible.
Pearson and Zhang \cite{PZ} proved that the adjacency tensor of a uniform hypergraph $G$ is weakly irreducible if and only if $G$ is connected.
Clearly, this shows that if $G$ is connected, then $\mathcal{A}(G), \mathcal{L}(G)$ and $\mathcal{Q}(G)$ are all weakly irreducible.

\begin{theorem}\label{PF} {\em \bf (The Perron-Frobenius Theorem for nonnegative tensors)}

$(1)$ {\em (Yang and Yang 2010 \cite{YY})} If $\T$ is a nonnegative tensor of order $k$ and dimension $n$, then $\rho(\T)$ is an $H^+$-eigenvalue of $\T$.

$(2)$ {\em (Friedland, Gaubert and Han 2013 \cite{FGH})} If furthermore $\T$ is weakly irreducible, then $\rho(\T)$ is the unique $H^{++}$-eigenvalue of $\T$,
with the unique eigenvector $x \in \mathbb{R}_{++}^{n}$, up to a positive scaling coefficient.

$(3)$ {\em(Chang, Pearson and Zhang 2008 \cite{CPZ})} If moreover $\T$ is irreducible, then $\rho(\T)$ is the unique $H^{+}$-eigenvalue of $\T$,
with the unique eigenvector $x \in \mathbb{R}_{+}^{n}$, up to a positive scaling coefficient.

\end{theorem}

For a connected $k$-uniform hypergraph $G$ (or a hypergraph $\Go$ obtained from $G$ by adding some loops),  by Theorem \ref{PF}, there exists a unique positive eigenvector of $\A(G)$ (respectively, $\Q(G)$), up to a scale, corresponding to its spectral radius, which is called the {\it Perron vector} of $\A(G)$ (respectively, $\Q(G)$).
The product $\A(G)x^k$ or $\Q(G)x^k$ has an interpretation as follows:
$$\A(G)x^k=\sum_{\{v_{i_1},v_{i_2},\ldots,v_{i_k}\} \in E(G)} kx_{v_{i_1}}x_{v_{i_2}} \ldots x_{v_{i_k}},\eqno(2.1)$$
$$\Q(G)x^k=\sum_{v \in V(G)}d_v x_v^k+\sum_{\{v_{i_1},v_{i_2},\ldots,v_{i_k}\} \in E(G)} kx_{v_{i_1}}x_{v_{i_2}} \ldots x_{v_{i_k}}.\eqno(2.2)$$
The eigenvector equation $\mathcal{A}(G)x^{k-1}=\lambda x^{[k-1]}$ could be interpreted as
$$ \lambda x_v^{k-1}= \sum_{\{v,v_2,v_3,\ldots, v_k\} \in E(G)} x_{v_2}x_{v_3} \cdots x_{v_k}, \mbox{~for each~} v \in V(G).\eqno(2.3)$$
The eigenvector equation $\mathcal{Q}(G)x^{k-1}=\lambda x^{[k-1]}$ could be interpreted as
$$ [\lambda-d_v] x_v^{k-1}= \sum_{\{v,v_2,v_3,\ldots, v_k\} \in E(G)} x_{v_2}x_{v_3} \cdots x_{v_k}, \mbox{~for each~} v \in V(G).\eqno(2.4)$$
It is known that
$$\lamin^\A(G) \le \min_{\|x\|_k=1}\A(G)x^k, ~ \lamin^\Q(G) \le \min_{\|x\|_k=1}\Q(G)x^k,$$
with equality if and only if $x$ is an eigenvector corresponding to $\lamin^\A(G)$ (respectively, $\lamin^\Q(G)$).

Denote  by $\Delta(G)$ (respectively, $\delta(G)$) the maximum degree (respectively, the minimum degree) of a hypergraph $G$.

\begin{lemma} \label{interlace}
Let $\T$ be a  tensor of order $k$ and dimension $n$, and let $\T[S]$ be a principle subtensor of $\T$ with $S \subsetneq [n]$.
Then
$$ \lamin(\T)\le \lamin(\T[S]),~ \lamax(\T[S]) \le \lamax(\T);$$
if, in addition, $\T$ is nonnegative and weakly irreducible, then $\rho(\T[S]) < \rho(\T)$.
\end{lemma}

\begin{proof}
Let $x$ (respectively, $y$) be an eigenvector of $\T[S]$ corresponding to $\lamin(\T[S])$ (respectively, $\lamax(\T[S])$) with $\|x\|_k=1$ (respectively, $\|y\|_k=1$).
Define a vector $\bar{x}$ (respectively, $\bar{y}$) of $\mathbb{R}^n$ such that $\bar{x}_i=x_i$ (respectively, $\bar{y}_i=y_i$) if $i \in S$ and $\bar{x}_i=0$ (respectively, $\bar{y}_i=0$) otherwise.
Then
$$\lamin(\T) =\min_{\|z\|_k=1}\T z^k \le \T\bar{x}^k=\T[S]x^k=\lamin(\T[S]),$$
and
$$\lamax(\T) =\max_{\|z\|_k=1}\T z^k \ge \T\bar{y}^k=\T[S]y^k=\lamax(\T[S]).$$

If $\T$ is nonnegative and weakly irreducible, then by Theorem \ref{PF} there exists a positive eigenvector $z$ of $\T$ corresponding to $\rho(\T)$, i.e. $\T z^{k-1}=\rho(\T)z^{[k-1]}$.
Also by the weak irreducibility of $\T$, there exists at least one $i \in S$ such that $\T_{ii_2i_3\ldots i_k}>0$ for some $i_t \notin S$, where $t \in \{2,3,\ldots,k\}$.
So, $\T[S] z[S]^{k-1}\lneq \rho(\T)z[S]^{[k-1]}$.
Now by \cite[Corollary 3.4]{KF}, we have $\rho(\T[S]) < \rho(\T)$.
\end{proof}

By Lemma \ref{interlace}, for a hypergraph $G$, by taking a vertex $u$ with $d_u=\delta(G)$, we get $\Q(G)[u]=\delta(G)$, and hence
 $\lamin^\Q(G) \le \delta(G)$.
 Similarly, if we take a vertex $w$ with $d_w=\Delta(G)$, then $\rho^\Q(G) \ge \Delta(G)$.
 Furthermore, considering a component $H$ of $G$ which contains the vertex $w$, then
 $\rho^\Q(G) \ge \rho^\Q(H) > \Delta(H)=\Delta(G)$ as $\Q(H)$ is weakly irreducible.
 The latter result has been shown in \cite{HQX} with a more accurate bound.
Here we use a unified method to deal with the bounds of $\lamin^\Q(G)$ and $\rho^\Q(G)$.

\begin{corollary} \label{boundQ}
Let $G$ be a $k$-uniform hypergraph or a hypergraph obtained from a $k$-uniform hypergraph by adding some loops.
Then $\rho^\Q(G)>\Delta(G)$ and $\rho^\L(G)>\Delta(G)$. If furthermore $k$ is even, then
$\lamin^\Q(G) < \delta(G)$ and $\lamin^\L(G) < \delta(G)$
\end{corollary}

\begin{proof}
Here we only prove the result for eigenvalues of signless Laplacian tensor. The corresponding result for Laplacian tensor can be discussed similarly.
Without loss of generality, let $e=\{1,2,\ldots,k\}$ be an edge of $G$.
Considering the $H$-eigenvalues of the principal subtensor $\Q(G)[e]$, by the eigenvector equation $\Q(G)[e]x^{k-1}=\la x^{[k-1]}$, where $x \in \mathbb{R}^{k}$, we have
$$ [\la-d_i]x_i^{k-1}=\Pi_{j \in [k] \backslash \{i\}}x_j, \mbox{~for~} i =1,2,\ldots,k.$$
If there exists some $i$ such that $x_i=0$, letting $j$ be such that $x_j \ne 0$, by the $j$th equation we have $\la=d_j$.
Otherwise, all $x_i$'s are nonzero; and multiplying both sides of the above equations over all $i$'s, we get that
$$ f(\la):=(\la-d_1)(\la-d_2)\cdots (\la-d_k)-1=0.$$
 If $e$ contains the vertex with maximum degree, then $f(\Delta(G))<0$.
 Noting that  $f(\la) \to +\infty$ when $\la \to +\infty$, so we have $\rho(\Q(G)[e])>\Delta(G)$.
Similarly, if $k$ is even and $e$ contains the vertex with minimum degree, then $f(\la) \to +\infty$ when $\la \to -\infty$, and $f(\delta(G))<0$,
which implies that $\lamin(\Q(G)[e])<\delta(G)$.
The result now follows by Lemma \ref{interlace}.
\end{proof}

Finally we introduce the generalized power hypergraphs defined in \cite{KF}.

\begin{definition} {\em \cite{KF}}
Let $G=(V,E)$ be a simple graph. For any $k \ge 3$ and $1 \le s \le k/2$, the generalized power of $G$, denoted by $G^{k,s}$, is defined as
the $k$-uniform hypergraph with the vertex set $\{\v: v \in V\} \cup \{\mathbf{e}: e \in E\}$, and the edge set
$\{\u \cup \v \cup \mathbf{e}: e=\{u,v\} \in E\}$, where $\v$ is an $s$-set containing $v$ and $\mathbf{e}$ is a $(k-2s)$-set corresponding to $e$.
\end{definition}

\begin{center}
\includegraphics[scale=.6]{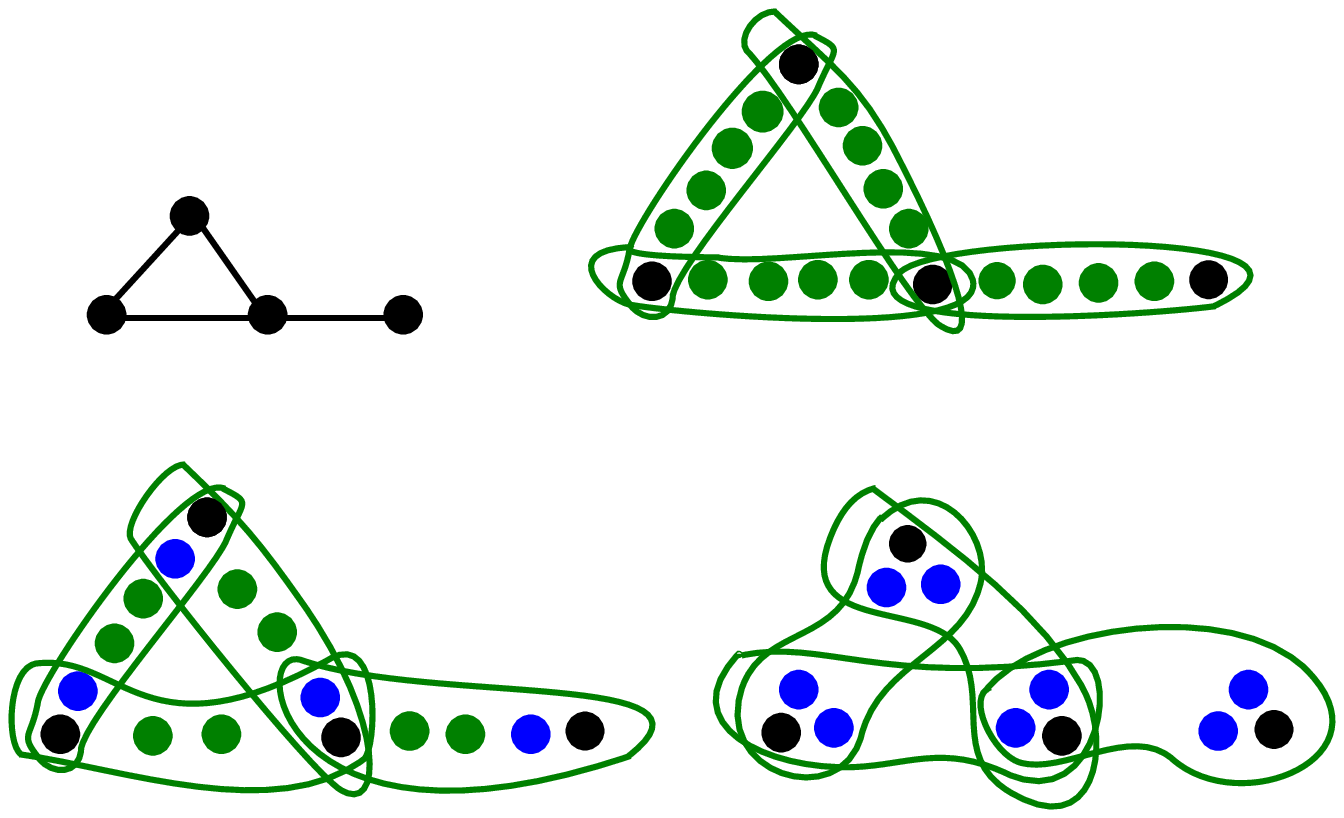}

\small{Fig. 2.1 (c.f. \cite{KF}) Constructing power hypergraphs $G^6$ (right upper), $G^{6,2}$ (left below) and $G^{6,3}$ (right below) from a simple graph $G$ (left upper), where a closed curve represents an edge}
\end{center}

Intuitively, $G^{k,s}$ is obtained from $G$ by replacing each vertex $v$ by an $s$-subset $\v$ and replacing each edge $\{u,v\}$ by a $k$-set obtained from $\u \cup \v$ by adding  $(k-2s)$ new vertices.
  If $s=1$, then $G^{k,s}$  is exactly the $k$th power hypergraph of $G$.
 When $G$ is a path or a cycle, then $G^{k,s}$ is an $s$-path or $s$-cycle for $s \le k/2$.
Note that if $s<k/2$, then $G^{k,s}$ is a cored hypergraphs and hence is odd-bipartite.
If $s=k/2$ ($k$ being even), then $G^{k,s}$  is obtained from $G$ by only blowing up its vertices.
In this case, $k$ is always assumed to be even; $\{u,v\}$ is an edge of $G$ if and only if $\u\cup\v$ is an edge of $G^{k,{k \over 2} }$, where we use the bold $\v$ to denote the blowing-up of the vertex $v$ in $G$.
For simplicity, we write $\u\v$ rather than $\u\cup\v$, and call $\u$ a {\it half edge} of $G^{k,{k \over 2} }$.
Denote by $d_\u$ the common degree of the vertices in $\u$.

If $G=\Go$, a simple graph with loops (i.e. edges containing only one vertex), then $(\Go)^{k,s}$ will have loops containing $k-s$ vertices.
In particular, $(\Go)^{k,{k \over 2}}$ will have loops containing ${k \over 2}$ vertices.
That is, if $\{u\}$ is a loop of $\Go$, then the half edge $\u$ is a loop of $(\Go)^{k,{k \over 2}}$.

\begin{lemma}\label{NOB} {\em \cite{KF}} Let $G$ be a simple graph without loops and let $k$ be an even positive integer.
The hypergraph $G^{k,{k \over 2}}$ is non-odd-bipartite if and only if $G$ is non-bipartite.
\end{lemma}

\section{Relationship between the eigenvalues of $G$ and $G^{k,\frac{k}{2}}$}
Let $G$ be a simple graph on $n$ vertices possibly with loops.
We first list some properties of the eigenvalues and eigenvectors of $\A(G^{k,\frac{k}{2}})$ and $\Q(G^{k,\frac{k}{2}})$.
First $\A(G^{k,\frac{k}{2}})$ has zero eigenvalues with geometric multiplicity at least $\frac{kn}{2}$ (the number of vertices of $G^{k,\frac{k}{2}}$).
Let $v$ be an arbitrary fixed vertex of $G^{k,\frac{k}{2}}$.
Define a vector $\x$ on $G^{k,\frac{k}{2}}$ such that $\x_v=1$ and $\x_u=0$ for any other vertices $u \ne v$.
It is easy to verify by (2.3) that $0$ is an eigenvalue of $\A(G^{k,\frac{k}{2}})$ with $\x$ as an eigenvector.
Similarly, also using the vector $\x$ defined as the above, by (2.4) we get that $d_v$ is an eigenvalue of $\Q(G^{k,\frac{k}{2}})$ with geometric multiplicity at least $\frac{k}{2}$ (the number of vertices in the half edge $\v$).

From the above facts, we find that the vertices in the same half edge of $G^{k,\frac{k}{2}}$ may have different values given by eigenvectors of $\A(G^{k,\frac{k}{2}})$ or $\Q(G^{k,\frac{k}{2}})$.
However, if $\la \ne 0$ (respectively, $\la \ne d_v$ for some vertex $v$) as an eigenvalue of  $\A(G^{k,\frac{k}{2}})$ (respectively, an eigenvalue of  $\Q(G^{k,\frac{k}{2}})$),
we will have the following property on the eigenvectors associated with $\la$.
For a nonempty subset $S \subseteq V(G^{k,{k \over 2}})$, denote $x^S:=\Pi_{v \in S} x_v$, where $x$ is a vector defined on the vertices of $G^{k,{k \over 2} }$.

\begin{lemma}
\label{abs Vec equal}Let $G$ be a simple graph possibly with loops.
Let $u$ and $\bar{u}$ be two vertices in the same half edge $\mathbf{u}$ of $G^{k,\frac{k}{2}}$.
If $\x$ is an eigenvector of $\A(G^{k,\frac{k}{2}})$ corresponding an eigenvalue $\lambda \ne 0$,
or an eigenvector of $\Q(G^{k,\frac{k}{2}})$ corresponding an eigenvalue $\lambda \ne d_\u$, then $\x_u^k=\x_v^k$, and hence $|\x_{u}|= |\x_{\bar{u}}| $.
\end{lemma}

\begin{proof}
If $\x$ is an eigenvector of $\A(G^{k,\frac{k}{2}})$, by the eigenvector equation (2.3),
$$ \lambda \x_u^{k-1}=\sum_{\v: \u\v \in E(G^{k,\frac{k}{2}})}\x^{\u\v \backslash \{u\}},
\lambda \x_{\bar{u}}^{k-1}=\sum_{\v: \u\v \in E(G^{k,\frac{k}{2}})}\x^{\u\v \backslash \{\bar{u}\}}.$$
So we have $\lambda \x_u^{k}=\lambda \x_{\bar{u}}^{k}$, which implies the result.
Similarly, if $\x$ is an eigenvector of $\Q(G^{k,\frac{k}{2}})$, by (2.4) we have $(\lambda-d_\u) \x_u^{k}=(\lambda-d_\u) \x_{\bar{u}}^{k}$.
As $\lambda \ne d_\u$,  the result also follows.
\end{proof}

If $G$ is a simple graph containing at least one edge, by Theorem \ref{PF} and Corollary \ref{boundQ},
$$\rho^\A(G^{k,\frac{k}{2}})>0, \rho^\Q(G^{k,\frac{k}{2}})>\Delta(G), \lamin^\Q(G^{k,\frac{k}{2}}) < \delta(G).$$
 Note that the sum of all eigenvalues of $\A(G^{k,\frac{k}{2}})$ is $(k-1)^{n-1}\mbox{tr}(\A(G^{k,\frac{k}{2}}))=0$ (\cite{Qi2}), which implies $\lamin^\A(G^{k,\frac{k}{2}})<0$, where $\mbox{tr}(\A)$ is the trace of $\A$.
Now, then by Lemma \ref{abs Vec equal}, all vertices in a half edge $\u$ of $G^{k,\frac{k}{2}}$ have the same modular given by the eigenvector $\x$ of $\rho^\A(G^{k,\frac{k}{2}}), \rho^\Q(G^{k,\frac{k}{2}}), \lamin^\A(G)$ or $\lamin^\Q(G)$; the common modular is denoted by $|\x_{\u}| $.

\begin{corollary}\label{nonzero-H}
Let $\la$ be an $H$-eigenvalue of $\A(G^{k,\frac{k}{2}})$ or $\Q(G^{k,\frac{k}{2}})$ corresponding to an eigenvector $\x$, where $G$ is a simple graph possibly with loops.
Suppose that $\x$ contains no zero entries.
If $\la \ne 0$ as an eigenvalue of $\A(G^{k,\frac{k}{2}})$ or $\la \notin \{d_\u: u \in V(G)\}$ as an eigenvalue of $\Q(G^{k,\frac{k}{2}})$,
then $\la$ is an eigenvalue of $\A(G)$ or $\Q(G)$ with eigenvector $x$ such that $x_v=\x^\v$ for each $v \in V(G)$.
\end{corollary}

\begin{proof}
First suppose that $\la$ is a nonzero $H$-eigenvalue of $\A(G^{k,\frac{k}{2}})$.
For each vertex $u \in \u$, by the eigenvector equation,
$$ \lambda \x_u^{k-1}=\sum_{\v: \u\v \in E(G^{k,\frac{k}{2}})}\x^{\u\v \backslash \{u\}}.$$
So $ \lambda \x_u^{k}=\sum_{\v: \u\v \in E(G^{k,\frac{k}{2}})}\x^{\u\v}$.
 Then $ \lambda \x^\u=\sum_{\v: \u\v \in E(G^{k,\frac{k}{2}})}\x^{\v}$ as $k$ is even and $\x_u^{k}=(\x^\u)^2 \ne 0$ by Lemma \ref{abs Vec equal}.
This implies that $\la x_u=\sum_{v: uv \in E(G)}x_v$, and $\la$ is an eigenvalue of $\A(G)$.
The discussion for the $\la$ being an eigenvalue of $\Q(G^{k,\frac{k}{2}})$ is similar.
\end{proof}

Now we discuss the general case that the $H$-eigenvector $\x$ in Corollary \ref{nonzero-H} may contain zero entries.
Let $\mathbf{U}$ be a union of half edges of $G^{k,\frac{k}{2}}$, where $G$ is a simple graph without loops.
Let $U=\{u: \u \subseteq \mathbf{U}\}$.
The tensor $\A(G^{k,\frac{k}{2}}[\mathbf{U}])$ (respectively, the matrix $\A(G)[U]$) is
exactly the adjacency tensor of $G^{k,\frac{k}{2}}[\mathbf{U}]$ (respectively, $G[U]$), the sub-hypergraph  of $G^{k,\frac{k}{2}}$ induced by $\mathbf{U}$ (respectively, the subgraph of $G$
induced by $U$).
However, $\Q(G^{k,\frac{k}{2}})[\mathbf{U}]$ (respectively, $\Q(G)[U]$) is not the signless Laplacian tensor of $G^{k,\frac{k}{2}}[\mathbf{U}]$ (respectively, $G[U]$).
The {\it modified induced subgraph} of a simple graph $G$ induced by the vertex set $U \subset V(G)$, denoted by $\Go[U]$, is
the induced subgraph $G[U]$ together with $d_v(G)-d_v(G[U])$ loops on each vertex $v \in U$.
Then the signless Laplacian matrix of $\Go[U]$ is exactly $\Q(G)[U]$, i.e. $\Q(\Go[U])=\Q(G)[U]$.
Therefore, $\Q(G^{k,\frac{k}{2}})[\mathbf{U}]=\Q(\Go[U]^{k,\frac{k}{2}})$.

\begin{theorem} \label{main1}
Let $\la$ be an $H$-eigenvalue of $\A(G^{k,\frac{k}{2}})$ (respectively, $\Q(G^{k,\frac{k}{2}})$) corresponding to an eigenvector $\x$,
where $G$ is a simple graph possibly with loops.
Suppose that $\la \ne 0$ as an eigenvalue of $\A(G^{k,\frac{k}{2}})$ or $\la \notin \{d_\u: u \in V(G)\}$ as an eigenvalue of $\Q(G^{k,\frac{k}{2}})$.
Let $\mathbf{U}=\cup\{\u: |\x_\u| > 0\}$ and $U=\{u: \u \subseteq \mathbf{U}\}$.
Then the following results hold.

{\em (1)} $G^{k,\frac{k}{2}}[\mathbf{U}]$ contains no isolated half edges, and hence $G[U]$ contains no isolated vertices.

{\em (2)} $\la$ is an $H$-eigenvalue of $\A(G^{k,\frac{k}{2}})[\mathbf{U}]$ (respectively, $\Q(G^{k,\frac{k}{2}})[\mathbf{U}]$) with eigenvector $\bar{\x}:=\x[\mathbf{U}]$.

{\em (3)} $\la$ is an eigenvalue of $\A(G[U])$ (respectively, $\Q(\Go[U])$) with the eigenvector $\bar{x}$ such that $\bar{x}_u=\bar{\x}^\u$ for each $u \in U$.
\end{theorem}

\begin{proof}
By (2.3) or (2.4), it is easy to verify the assertion (1) or (2).
Note that $\A(G^{k,\frac{k}{2}})[\mathbf{U}]=\A(G[U]^{k,\frac{k}{2}})$ and
$\Q(G^{k,\frac{k}{2}})[\mathbf{U}]=\Q(\Go[U]^{k,\frac{k}{2}})$, the assertion (3) follows from Lemma \ref{nonzero-H} as $\bar{\x}$ contains no zero entries.
\end{proof}

\begin{lemma}\label{contain} Let $G$ be a simple graph possibly with loops.
Each eigenvalue of $\A(G)$ (respectively, $\Q(G)$) is an $H$-eigenvalue of $\A(G^{k,\frac{k}{2}})$ (respectively, $\Q(G^{k,\frac{k}{2}})$).
\end{lemma}

\begin{proof}
For each vertex $v$ of $G$, we assume that $v$ is also contained in the corresponding half edge $\v$ in $G^{k,\frac{k}{2}}$.
Let $x$ be an eigenvector of $\A(G)$ corresponding to an eigenvalue $\lambda$.
Let $\x$ be a vector defined on $G^{k,\frac{k}{2}}$ as follows:
$$ \x_v=\sgn(x_v)|x_v|^{2/k}, \x_{\bar{v}}=|x_v|^{2/k}, \mbox{~for each vertex~} \bar{v} \in \v \backslash \{v\} \mbox{~and each~} v \in V(G).\eqno(3.1)$$
Then $$\x^{\v}=x_v, \mbox{~for each~} v \in V(G).\eqno(3.2)$$
Noting that $\lambda x_v=\sum_{uv \in E(G)}x_u$, it is easy to verify that
$$ \lambda \x_v^{k-1}=\sum_{\u\v \in E(G^{k,\frac{k}{2}})}\x^{\u\v \backslash \{v\}} \mbox{~and~}
\lambda \x_{\bar{v}}^{k-1}=\sum_{\u\v \in E(G^{k,\frac{k}{2}})}\x^{\u\v \backslash \{\bar{v}\}} \mbox{~for each~}  \bar{v} \in \v \backslash \{v\}. $$
So, $\lambda$ is also an eigenvalue of $\A(G^{k,\frac{k}{2}})$.

Similarly, if $x$ is an eigenvector of $\Q(G)$ corresponding to an eigenvalue $\lambda$, then $(\lambda-d_v) x_v=\sum_{uv \in E(G)}x_u$.
Defining a vector $\x$ as in (3.1), we get $ (\lambda-d_v) \x_v^{k-1}=\sum_{\u\v \in E(G^{k,\frac{k}{2}})}\x^{\u\v \backslash \{v\}}$ for each vertex $v$ of $G^{k,\frac{k}{2}}$.
\end{proof}

A simple graph consisting of only one vertex possibly with loops is considered to be connected.

\begin{theorem}\label{main2}
Let $G$ be a simple graph without loops.
Then, regardless of multiplicities, the $H$-spectrum of $\A(G^{k,\frac{k}{2}})$ (respectively, $\Q(G^{k,\frac{k}{2}})$)
 consists of all eigenvalues of the adjacency matrices (respectively, the signless Laplacian matrices) of the connected induced subgraphs (respectively, modified induced subgraphs) of $G$.
\end{theorem}

\begin{proof}
Suppose that $\la$ is an $H$-eigenvalue of $\A(G^{k,\frac{k}{2}})$.
If $\la =0$, then $\la$ is the eigenvalue of adjacency matrix of an isolated vertex.
Assume that $\la \ne 0$.
By Theorem \ref{main1}, $\la$ is an eigenvalue of the adjacency matrix of some induced subgraph of $G$.
Surely, $\la$ is an eigenvalue of the adjacency matrix of some connected induced subgraph of $G$.

Conversely, if $\la $ is an eigenvalue of the adjacency matrix of some connected induced subgraph $G[U]$ with $x$ as a corresponding eigenvector,
then by Lemma \ref{contain}, $\la$ is an eigenvalue of $G[U]^{k,\frac{k}{2}}=G^{k,\frac{k}{2}}[\mathbf{U}]$ with eigenvector $\x$ as defined in (3.1),
where $\mathbf{U}=\cup\{\u: u \in U\}$.
Extending the eigenvector $\x$ defined on $G^{k,\frac{k}{2}}[\mathbf{U}]$ to $G^{k,\frac{k}{2}}$ by assigning zeros to the vertices outside $\mathbf{U}$,
we will get a vector $\mathbf{y}$.
It is easy to verify by (2.3) that $\mathbf{y}$ is an eigenvector of $\A(G^{k,\frac{k}{2}})$ corresponding the eigenvalue $\la$.

Next assume that $\la$ is an $H$-eigenvalues of $\Q(G^{k,\frac{k}{2}})$.
If $\la=d_\v$ for some half edge $\v$, then $\la$ is eigenvalue of the signless Laplacian matrix of the modified subgraph induced by the isolated vertex $v$.
If $\la \notin \{d_\u: u \in G\}$, by Theorem \ref{main1}, $\la$ is an eigenvalue of the signless Laplacian matrix of some connected modified induced subgraph of $G$.
Conversely, if $\la $ is an eigenvalue of $\Q(\Go[U])$ with $x$ as a corresponding eigenvector,
by Lemma \ref{contain}, $\la$ is an eigenvalue of $\Q(\Go[U]^{k,\frac{k}{2}})=\Q(G^{k,\frac{k}{2}})[\mathbf{U}]$ with eigenvector $\x$ as defined in (3.1).
The remaining discussion is similar.
\end{proof}

\begin{corollary}\label{LE relation}
Let $G$ be a simple graph.
Then $\lamin^\A(G)=\lamin^\A(G^{k,\frac{k}{2}})$ and $\lamin^\Q(G)=\lamin^\Q(G^{k,\frac{k}{2}})$.
\end{corollary}

\begin{proof}
By the interlacing theorem of the eigenvalues of real symmetric matrices (see \cite[Chapter 4]{hc}), $\lamin^\A(G)$ (respectively, $\lamin^\Q(G)$) is the minimum of all least eigenvalues of the principal submatrices of $\A(G)$ (respectively, $\Q(G)$).
The result follows from Theorem \ref{main2}.
\end{proof}

\vspace{3mm}

\noindent{\scshape  Remark}:
(1) The relationship between a simple graph $G$ and $G^{k,s}$ for a general $s < k/2$ has been discussed by Yuan, Shao and Qi \cite{YQS}.
An interesting result is that $\rho^\A(G^{k,s})=\rho^\A(G)^{\frac{2s}{k}}$; see \cite{YQS}.

(2) We give two examples to illustrate Theorem \ref{main2}.
Denote by $P_n$ a simple path on $n$ vertices.
Then $P_2^{k,\frac{k}{2}}$ is a hypergraph on vertices $1,2,\ldots,k$ with only one edge $\{1,2,\ldots,k\}$.
It is known that the spectrum of $\A(P_2)$ is $\{1,-1\}$, and the spectrum of $\Q(P_2)$ is $\{0,2\}$.
We now compute the eigenvalues of $\A(P_2^{k,\frac{k}{2}})$ and $\Q(P_2^{k,\frac{k}{2}})$.
Let $\la$ be an eigenvalue of $\A(P_2^{k,\frac{k}{2}})$ corresponding to an eigenvector $x$.
Then by (2.3), for $i=1,2,\ldots,k$,
$ \la x_i^{k-1}=\Pi_{j \in [k],j \ne i} x_j$.
Just like the discussion in Corollary \ref{boundQ}, if $x_i=0$ for some $i$, then $\la=0$;
otherwise, $\la^k=1$.
Each $\la_j=e^{\frac{2\pi j}{k}\i} ~(j \in [k], \i^2=-1)$ is an eigenvalue of $\A(P_2^{k,\frac{k}{2}})$ with the eigenvector $\x^{(j)}$ defined as following: for any chosen $j$-subset $S$ of $[k]$,
$\x^{(j)}_i=e^{\frac{2\pi }{k}\i}$ if $i \in S$, and $\x^{(j)}_i=1$ otherwise.
So the eigenvalues of $\A(P_2^{k,\frac{k}{2}})$ are $0, \la_{k}=1,\la_{k/2}=-1$, and $\la_j$, $j \in [k]\backslash \{k/2,k\}$.
As $\Q(P_2^{k,\frac{k}{2}})=\mathcal{I}+\A(P_2^{k,\frac{k}{2}})$, the eigenvalues of $\Q(P_2^{k,\frac{k}{2}})$ are $1$ (the degree), $2$, $0$, and $1+\la_j$, $j \in [k]\backslash \{k/2,k\}$.
The real eigenvalue $\la_{k/2}=-1$ of $\A(P_2^{k,\frac{k}{2}})$ (respectively, the zero eigenvalue of $\Q(P_2^{k,\frac{k}{2}})$)  have two eigenvectors: one is an $N$-eigenvectors $\x^{(k/2)}$, the other is an eigenvector $y$ defined by $y_i=-1$ for an arbitrary fixed $i$ and $y_j=1$ for all $j \ne i$.

It is known that the spectrum of $\A(P_3)$ is $\{\sqrt{2},0,-\sqrt{2}\}$, and the spectrum of $\Q(P_3)$ is $\{0,1,3\}$.
The connected induced subgraphs of $P_3$ are $P_1,P_2,P_3$.
By Theorem \ref{main2}, the $H$-eigenvalues of $\A(P_3^{k,\frac{k}{2}})$ are $0$, $-1$, $1$, $\sqrt{2}$, $-\sqrt{2}$.
The connected modified induced subgraphs of $P_3$ are a vertex with one loop, a vertex with two loops, an edge with one loop on some vertex, and $P_3$.
The corresponding signless Laplacian matrices are $[1], [2], \left[\begin{array}{cc} 2 & 1 \\ 1 & 1 \end{array}\right], \Q(P_3)$.
So the $H$-eigenvalues of $\Q(P_3^{k,\frac{k}{2}})$ are $0,1,2,3, \frac{3 \pm \sqrt{5}}{2}$.

\section{Limit points of the least adjacency $H$-eigenvalues}
Hoffman \cite{Hof} observed if a simple graph $G$ properly contains a cycle, then $\rho(\A(G)) > \tau^{1/2}+\tau^{-1/2}=\tau^{3/2}=\sqrt{2+\sqrt{5}}$, where $\tau=(\sqrt{5}+1)/2$ is the golden mean.
He proved that $\tau^{3/2}$ is a limit point, and found all limit points of the adjacency spectral radii less than $\tau^{3/2}$.
The work of Hoffman was extended by Shearer \cite{S} to show that every real number $r \ge \tau^{3/2}$ is the limit point of the adjacency spectral radii of simple graphs.
Furthermore, Doob \cite{Doob} proved that for each $r \ge \tau^{3/2}$ (respectively, $r \le -\tau^{3/2}$) and for any $k$, there exists a sequence of graphs whose $k$th largest eigenvalues (respectively, $k$th smallest eigenvalues) converge to $r$.

The smallest limit point of the adjacency spectral radii of simple graphs is $2$, which is realized by a sequence of paths.
If $r < \tau^{3/2}$ is a limit point, it suffices to consider the trees by Hoffman's observation.
The construction of graphs whose adjacency spectral radii converge to $r \ge \tau^{3/2}$ in \cite{Doob,S} are trees $T(n_1,n_2,\ldots,n_k)$ called {\it caterpillars}, which is obtained from a path on vertices $v_1,v_2,\ldots,v_k$ by attaching $n_j \ge 0$ pendant edges at the vertex $v_j$ for each $j=1,2,\ldots,k$.

As the adjacency spectrum of a tree is symmetric with respect to the origin, the minus of the limit points of the spectral radius are the limit points of the least eigenvalue.
Since $-\rho^\A(T)=\lamin^\A(T)=\lamin^\A(T^{k,{k \over 2}})$ by Corollary \ref{LE relation}, we get the following result on the limit points of the least adjacency $H$-eigenvalues of $k$-uniform hypergraphs.

\begin{theorem}
For $n=1,2,\ldots$, let $\beta_n$ be the positive root of $P_n(x)=x^{n+1}-(1+x+x^2+\cdots+x^{n-1})$.
Let $\alpha_n=\beta_n^{1/2}+\beta_n^{-1/2}$. Then
$-2=-\alpha_1 >- \alpha_2 > \cdots $ are all limit points of the least $H$-eigenvalues of the adjacency tensor of hypergraphs greater than $-(\tau^{1/2}+\tau^{-1/2})=-\lim_n \alpha_n$.
\end{theorem}

\begin{theorem}
Each real number $r \le -\tau^{3/2}$ is a limit point of the least $H$-eigenvalue of the adjacency tensor of hypergraphs.
\end{theorem}

Finally we will construct a sequence of non-bipartite graphs $G$ whose least adjacency eigenvalues converge to $-\tau^{3/2}$.
Consequently we get sequence of non-odd-bipartite hypergraphs $G^{k,k/2}$ whose least adjacency eigenvalues converge to $-\tau^{3/2}$.
Denote by $C_n+e$ the simple graph obtained from a cycle $C_n$ on $n$ vertices by appending a pendant edge $e$ at some vertex.

\begin{lemma}\label{cn+e}
$\lim \limits_{n\rightarrow \infty }\lamin^\A(C_{2n+1}+e)=-\tau ^{\frac{3}{2}}$
\end{lemma}

\begin{proof}
Label the vertices of $C_{2n+1}+e$ as follows:
the pendant vertex is labeled by $v_0$, starting from the vertex of degree $3$ the vertices of the cycle are labeled by $v_1, v_2, \ldots, v_{2n+1}$ clockwise.
Note that now $e=v_0v_1$.
Denote $T_{2n+1}:=C_{2n+1}+e-v_{n+1}v_{n+2}$.
Let $x$ be a unit vector corresponding to the least eigenvalue of $C_{2n+1}+e$.
By symmetry, $x_{v_k}=x_{v_{2n+3-k}}$ for $k=2,3,\ldots,n+1$; in particular $x_{v_{n+1}}=x_{v_{n+2}}$.
So
$$\lamin^\A(C_{2n+1}+e)= \sum_{uv \in E(C_{2n+1}+e)} 2x_ux_v = x^T A(T_{2n+1})x + 2 x_{v_{n+1}}x_{v_{n+2}}
> \lamin^\A(T_{2n+1}).$$

On the other hand, let $y$ be  a unit vector corresponding to the least eigenvalue of $T_{2n+1}$.
Also by symmetry, $y_{v_{n+1}}=y_{v_{n+2}}$.
As $T_{2n+1}$ is bipartite, $\rho^A(T_{2n+1})=-\lamin^A(T_{2n+1})$ and $|y|$ is the Perron vector of $\A(T_{2n+1})$.
In addition, as shown by Hoffman \cite{Hof}, $\rho^A(T_{2n+1})$ increasingly converges to $\tau^{3/2}>2$.
So, for sufficiently large $n$, $\rho^\A(T_{2n+1})>2$.
By a discussion similar in \cite[Lemma 4.7]{KF}, we get that $|y_{v_1}| > |y_{v_2}| > \cdots > |y_{v_{n+1}}|$.
Note that
$$ 1=\sum_{i=0}^{2n+1}y_{v_i}^2 > y_{v_1}^2 +2 (y_{v_2}^2+  \cdots + y_{v_{n+1}}^2) > (2n+1)y_{v_{n+1}}^2.$$
So $2 y_{v_{n+1}}^2 < \frac{2}{2n+1}$.
As $y_{v_{n+1}}=y_{v_{n+2}}$, we have
$$
\lamin^\A(C_{2n+1}+e)  = \sum_{uv \in E(C_{2n+1}+e)} 2y_uy_v
= y^T \A(T_{2n+1})y + 2 y_{v_{n+1}}y_{v_{n+2}}
 < \lamin^\A(T_{2n+1})+\frac{2}{2n+1}.
$$
By the above discussion, for sufficiently large $n$,
$$\lamin^\A(T_{2n+1})<\lamin^\A(C_{2n+1}+e)<\lamin^\A(T_{2n+1})+\frac{2}{2n+1}.$$
So
$$\lim_{n \to \infty} \lamin^\A(C_{2n+1}+e) =\lim_{n \to \infty} \lamin^\A(T_{2n+1})=-\tau^{3/2}.$$
\end{proof}

By Lemma \ref{NOB}, Lemma \ref{cn+e} and Corollary \ref{LE relation}, we get the following result.

\begin{corollary}
$-\tau ^{\frac{3}{2}}$ is a limit point of the least $H$-eigenvalue
of adjacency tensor of non-odd-bipartite hypergraphs.
\end{corollary}

\end{document}